\documentclass{article}
\usepackage[utf8]{inputenc}
\usepackage{amsmath, amsthm, amssymb, enumerate, multicol, chngcntr}
\usepackage{xcolor}
\usepackage{pdfpages}
\usepackage{comment}

\theoremstyle{plain}
\newtheorem{thm}{Theorem}[section]

\newtheorem{prop}[thm]{Proposition}
\newtheorem{lemm}[thm]{Lemma}
\newtheorem{rema}[thm]{Remark}
\newtheorem{conj}[thm]{Conjecture}

\title{Enumeration of groups in some special varieties of $A$-groups}
\author{Arushi\footnote{\mbox{Dr. B. R. Ambedkar University Delhi, Delhi 110006, India;  E-mails: arushi.18@stu.aud.ac.in,}  arushi.garvita@gmail.com.}  \ and \  Geetha Venkataraman \footnote{Corresponding Author, Dr. B. R. Ambedkar University Delhi, Delhi 110006, India;  E-mails: geetha@aud.ac.in, geevenkat@gmail.com.}}

\begin{document}
\maketitle

\noindent
{\small{\bf ABSTRACT:}}
We find an upper bound for the number of groups of order $n$ up to isomorphism in the variety ${\mathfrak{S}}={\mathfrak{A}_p}{\mathfrak{A}_q}{\mathfrak{A}_r}$ where $p$, $q$ and $r$ are distinct primes. We also find a bound  on the orders and on the number of conjugacy classes of subgroups that are maximal amongst the subgroups of the general linear group that are also in the variety $\mathfrak{A}_q\mathfrak{A}_r$. 

\noindent
{\small{\bf Keywords}{:} }
group enumeration, variety of groups, general linear group, symmetric group, transitive subgroup, primitive subgroup, conjugacy class.

\noindent
{\small{\bf AMS Subject Classification}{:} }
20B15, 20B35, 20D10, 20E10, 20E45, 20H30. 

\vspace{.2in}
\noindent
{\bf THIS IS AN EARLY VERSION OF THE PAPER. FOR THE FINAL
VERSION SEE \url{https://doi.org/10.1017/S0004972724000431}.}
\vspace{.2in}

\baselineskip=\normalbaselineskip
\section{Introduction}
A group is an $A$-group if its nilpotent subgroups are abelian. Let ${\mathfrak{B}}$ be any class of groups then the number of groups of order $n$ up to isomorphism is denoted by $f_{\mathfrak{B}}(n)$. Computing $f(n)$ becomes harder as $n$ gets bigger. Thus in the area of group enumerations, we attempt to approximate $f(n)$. When counting is restricted to the class of abelian groups, $A$-groups, and groups in general, then $f(n)$ behaves differently asymptotically. Let $f_{A,sol}(n)$ be the number of isomorphism classes of soluble $A$-groups of order $n$. In \cite{GD1969} G.A.Dickenson showed that $f_{A,sol}(n) \leq n^{c \log{(n)}}$ for some constant $c$. In \cite{Annabel1987} McIver and Neumann showed that the number of non-isomorphic $A$-groups of order $n$ is at most $n^{{\lambda}+1}$ where $\lambda$ is the number of prime divisors of $n$ including multiplicities. In the same paper they stated the following conjecture based on a result of Higman \cite{GH1960} and Sims \cite{CS1965} on $p$-group enumerations.
\begin{conj} 
\label{conj_1} 
Let $f(n)$ be the number of (isomorphism classes of groups of) order $n$. Then\\
$$ f(n) \leq n^{(\frac{2}{27} + \epsilon){\lambda}^2}$$ 
where $\epsilon \to 0$ as $\lambda \to \infty$.
\end{conj}
 In 1993 L.Pyber \cite{LP1993} proved a powerful version of \ref{conj_1}. He proved that the number of groups of order $n$ with specified Sylow subgroups is at most $n^{75 \mu + 16}$ where $\mu$ is the largest integer such that $p^{\mu}$ divides $n$ for some prime $p$. Using Higman-Sims and Pyber's result we get that $ f(n) \leq n^{\frac{2}{27}{\mu}^2 + O(\mu^{5/3})}$. In \cite {GV1997}, it was shown that $f_{A,sol}(n) \leq n^{7\mu + 6}$.

The variety ${\mathfrak{A}_u}{\mathfrak{A}_v}$ consists of all groups $G$ with an abelian normal subgroup $N$ of exponent dividing $u$ such that $G/N$ is abelian of exponent dividing $v$. For more on varieties see \cite{HN1967}. Let $p, q$ and $r$ be distinct primes. In this paper we find a bound for $f_{\mathfrak{S}}(n)$ where $\mathfrak{S}={\mathfrak{A}_p}{\mathfrak{A}_q}{\mathfrak{A}_r}$ and $f_{\mathfrak{S}}(n)$ counts the groups in $\mathfrak{S}$ of order $n$ up to isomorphism. The idea behind studying the variety $\mathfrak{S}$ is that enumerating within the varieties of $A$-groups might yield a better upper bound for the enumeration function for $A$-groups. The `best' bounds for $A$-groups, or even solvable $A$-groups, are still lacking the correct leading term. It is believed that a correct leading term for the upper bound of $A$-groups would lead to the right error term for the enumeration of groups in general.

A few smaller varieties of $A$-groups have already been studied \cite[Chapter 18]{SPG2007}. The class of $A$-groups for which the `best' bounds exist was obtained by enumerating in such small variety of $A$-groups, but this did not narrow the difference between the upper and lower bounds for $f_{A,sol}(n)$. The analysed groups did not contribute a large enough collection of $A$-groups. Hence a good lower bound could not be reached. In order to reduce the difference, we enumerate in a larger variety of $A$-groups, namely $\mathfrak{S}$.

Throughout the paper, $p$, $q$, $r$ and $t$ are distinct primes. We assume that $s$ is a power of $t$. We take logarithms to the base $2$ unless stated otherwise and follow the convention that $0 \in \mathbb{N}$. We use $C_m$ to denote a cyclic group of order $m$ for any positive integer $m$. Let $O_{p'}{(G)}$ denote the largest normal $p'$-subgroup of $G$. The techniques we use are similar to those in \cite{LP1993}, \cite{SPG2007} and \cite{GV1997}. The main results proved in this paper are as follows. 

\medskip
\noindent
{\bf{Theorem A}} Let $n=p^{\alpha}q^{\beta}r^{\gamma}$ where $\alpha, \beta, \gamma \in \mathbb{ N}$. Then\\
$$f_{\mathfrak{S}}(n) \leq p^{6\alpha^{2}} \,2^{\alpha - 1 + (23/6) \alpha \log(\alpha) + \alpha \log(6)}\, (6^{1/2})^{{({\alpha + \gamma})\beta} + {({\alpha + \beta})\gamma} +{\alpha(\alpha-1)/2}}\, n^{\beta + \gamma}.$$

In order to prove Theorem A, we prove a bound on the number of conjugacy classes of subgroups that are maximal amongst subgroups of ${\rm GL}(\alpha, s)$ and which are in the variety $\mathfrak{A}_q \mathfrak{A}_r$ or $\mathfrak{A}_r$. We also prove results about the order of primitive subgroups of $S_n$ which are in the variety $\mathfrak{A}_q \mathfrak{A}_r$ and show that they form a single conjugacy class. These results are mentioned below.

\medskip
\noindent
{\bf{Theorem B}} 
 Let $q$ and $r$ be distinct primes. Let $G$ be a primitive subgroup of $S_n$ which is in $\mathfrak{A}_q \mathfrak{A}_r$ and let  $|G| = q^\beta r^\gamma$ where $\beta, \gamma \in \mathbb{N}$. Let $M$ be a minimal normal subgroup of $G$. 
\begin{enumerate}[(i)]
\item If $\beta = 0$, then $|M|$ is a power of $r$ and $|G| = n = r$ with $G \cong C_r$. 
\item If $\beta \geq 1$ then $|M| = q^{\beta} = n$ with $ \beta = \mbox{order } q \bmod r$. Further $G \cong M \rtimes C_r$ and $|G| = nr < n^{2}$.
\item If $\gamma = 0$, then $|M|$ is a power of $q$ and $|G| = n = q$ with $G\cong C_q$. 
\end{enumerate}

\medskip
\noindent
{\bf{Theorem C}} The primitive subgroups of $S_n$ which are in $\mathfrak{A}_q \mathfrak{A}_r$ and of order $q^{\beta}r^{\gamma}$ where $\beta, \gamma \in \mathbb{ N}$, form a single conjugacy class.

\medskip
\noindent
{\bf{Theorem D}}
There exist constants $b$ and $c$ such that the number of conjugacy classes of subgroups that are maximal amongst the subgroups of ${\rm GL}(\alpha, s)$ which are in $\mathfrak{A}_q\mathfrak{A}_r$ is at most  $$
2^{(b+c) ({\alpha}^2/\sqrt{\log (\alpha)}) + (5/6) \alpha \log (\alpha) + \alpha (1+ \log(6))} s^{(3+c){\alpha}^2}
$$
where $t, q$ and $r$ are distinct primes, $s$ is a power of $t$ and $\alpha > 1$.

\smallskip
Section \ref{PrimSymsub} investigates primitive subgroups of $S_{n}$ which are in $\mathfrak{A}_r$ or $\mathfrak{A}_q\mathfrak{A}_r$. Sections \ref{Genlinsub_r} and \ref{Genlinsub} deal with subgroups of the general linear group. Theorem A  is proved in section \ref{Enu_pqr}.

\section{Primitive subgroups of $S_{n}$ which are in $\mathfrak{A}_r$ or $\mathfrak{A}_q\mathfrak{A}_r$}
\label{PrimSymsub}
In this section we prove results which give us the structure of the primitive subgroups of $S_{n}$ which are in $\mathfrak{A}_r$ or $\mathfrak{A}_q\mathfrak{A}_r$. We also show that such subgroups form a single conjugacy class. Both Theorems B and C are proved in this section.

Theorem B provides the order of a primitive subgroup of $S_n$ which is in the variety $\mathfrak{A}_q\mathfrak{A}_r$. By \cite[Proposition 2.1]{GV1997}, we know that if $G$ is a soluble $A$-subgroup of $S_n$ then $|G| \leq (6^{1/2})^{n-1}$. Indeed, this bound is determined primarily by considering primitive soluble $A$-subgroups of $S_n$. This bound would clearly hold for any subgroup of $S_n$ which is in the variety $\mathfrak{A}_q\mathfrak{A}_r$. However, we show that when the subgroup is primitive and in the variety $\mathfrak{A}_q\mathfrak{A}_r$ we can do better. 
\begin{lemm}
\label{prim_sym_sub_r}
  $S_n$ has a primitive subgroup in $\mathfrak{A}_r$ if and only if $n=r$. In this case, any primitive subgroup $G$ which is in $\mathfrak{A}_r$ will be cyclic of order $r$. All primitive subgroups of $S_n$  which are in $\mathfrak{A}_r$ form a single conjugacy class.
\end{lemm}
 
\begin{proof}
Let $G$ be a primitive subgroup of $S_n$ which is in $\mathfrak{A}_r$. Since $G$ is soluble therefore, $M$ is an elementary abelian $r$-subgroup. By the O'Nan-Scott Theorem \cite{LS1981}, we get that $|M| = n = |G|$. So $G = M \cong C_{r}$ and $n=r$. Conversely, any transitive subgroup $G$ of $S_{r}$ is primitive \cite[Theorem 8.3]{HW1964}. Since $n$ is prime, any subgroup of order $n$ in $S_n$ will be generated by a $n$-cycle. Further, any two $n$-cycles are conjugate in $S_n$. Thus the primitive subgroups of $S_n$ that are also in $\mathfrak{A}_r$ form a single conjugacy class.\end{proof}

\noindent {\bf{Proof of Theorem B}}
\begin{proof} Let $G$ be a subgroup of $S_{\Omega}$, where $|\Omega|=n$ and let $G \in \mathfrak{A}_q\mathfrak{A}_r$. Then $G = Q \rtimes R$ where $Q$ is an elementary abelian Sylow $q$-subgroup, $R$ is an elementary abelian Sylow $r$-subgroup and $|G| = q^\beta r^\gamma$ where $\beta$, $\gamma \in \mathbb{N}$. Let $M$ be a minimal normal subgroup of $G$. Then $M$ is an elementary abelian $u$-group. Clearly $|M|=u^{k}$ for some $k > 1$ and for some prime $u \in \{q,r\}$.
  
Now $F(G)$, the Fitting subgroup of $G$ is an abelian normal subgroup of $G$ and so, then by the O'Nan-Scott Theorem,  $n=|M|=|F(G)|$. But $M \leq F(G)$, therefore, $M=F(G)$ and $n=u^k$. If $\beta \geq 1$, then $Q \leq F(G)$ and we have $n= q^{\beta} = u^k$ and $M=F(G) = Q$. Let $H=G_{\alpha}$ be the stabiliser of an $\alpha \in \Omega$. By \cite[Proposition 6.13]{SPG2007}, we get that $G$ is a semidirect product of $M$ by $H$ and that $H$ acts faithfully by conjugation on $M$. By Maschke's theorem, $M$ is completely reducible. But $M$ is a minimal normal subgroup of $G$, therefore, $M$ is a non-trivial irreducible $\mathbb{F}_{q}H$-module and $H$ is an abelian group acting faithfully on $M$. So by \cite[Cor. 4.1]{GV1999},  we get $H\cong C_{r}$ and $ \beta = \dim M = \mbox{order } q \bmod r$ and the result follows. If $\gamma = 0$ or $\beta = 0$ then $|G|$ is a power of $u$ where $u \in \{q, r\}$. Thus $G$ is a primitive subgroup which is also in $\mathfrak{A}_u$. So the result follows by Lemma \ref{prim_sym_sub_r}. \end{proof}

It is clear from the above results that if $S_n$ has a primitive subgroup $G$ of order $q^{\beta}r^{\gamma}$ in $\mathfrak{A}_q\mathfrak{A}_r$ then $n$ has to be $r$ or $q$ and $G$ is cyclic with $|G|=n$ or we must have that $n=q^{\beta}$ and $G$ is a semi-direct product of an elementary abelian $q$-group of order $q^{\beta}$ by a cyclic group of order $r$. The limits imposed on $n$ and on the structure of such primitive subgroups gives us the next result.

\medskip
\noindent {\bf{Proof of Theorem C}}

\begin{proof}Let $G$ be a primitive subgroup of $S_{\Omega}$ which is in $\mathfrak{A}_q\mathfrak{A}_r$, where $|\Omega| = n$ and let $|G|= q^{\beta}r^{\gamma}$. Let $M$ be a minimal normal subgroup of $G$. As seen in the proof of Theorem B we get that $M = F(G)$ and $n=|M|$, is either a power of $q$ or $r$. If $\gamma = 0$ or $\beta = 0$ then $|G|$ is a power of $u$ where $u \in \{q, r\}$. Thus $G$ is a primitive subgroup which is also in $\mathfrak{A}_u$. So the result follows by Lemma \ref{prim_sym_sub_r}. 

We know the structure of $G$ when $\beta \geq 1$ from the proof of Theorem B. Hence $H$ can be regarded as a soluble $r$-subgroup of ${\rm GL}(\beta,q)$ and it is not difficult to show that the conjugacy class of $G$ in $S_{n}$ is determined by the conjugacy class of $H$ in ${\rm GL}(\beta,q)$. Let $S$ be a Singer subgroup of ${\rm GL}(\beta,q)$. So  $|S| = q^{\beta}-1$. Now $|H| =r$ and $r$ divides $|S|$. Further $\gcd (|{\rm GL}(\beta,q)|/|S|, r)=1$ as $\beta$ is the least positive integer such that $r \mid q^{\beta}-1$. Using \cite[Theorem 2.11]{MH1970}, we get that $H^{x} \leq S$ for some $x \in {\rm GL}(\beta,q)$. Since all Singer subgroups are conjugate in ${\rm GL}(\beta,q)$ the result follows.\end{proof}

\section{Subgroups of ${\rm GL}(\alpha, s)$ which are in $\mathfrak{A}_r$}
\label{Genlinsub_r}
In this section we prove results which give us a bound on the number of conjugacy classes of the subgroups that are maximal amongst subgroups of ${\rm GL}(\alpha, s)$ that are in $\mathfrak{A}_r$. The limits on the structure of such groups ensures that if they exist, they form a single conjugacy class. 
\begin{lemm}
\label{GL_conj_irr_sub_r}
The number of conjugacy classes of irreducible subgroups of ${\rm GL}(\alpha, s)$ which are also in $\mathfrak{A}_r$ is at most $1$.
\end{lemm}
\begin{proof}
Let $G$ be a non-trivial irreducible subgroup of ${\rm GL}(\alpha, s)$ which is also in $\mathfrak{A}_r$. Then $G$ is an elementary abelian $r$-group of order $r^{\gamma}$, say, where $\gamma \in \mathbb{N}$. Since $G$ is a faithful abelian irreducible subgroup of ${\rm GL}(\alpha, s)$ whose order is coprime to $s$ we know that $G$ is cyclic (\cite[Lemma 4.2]{GV1999}). Thus $|G|=r$ and $\alpha = d$, where $d= \mbox{order } s \bmod r$. Using \cite[Theorem 2.3.3]{MWS1992}, we get that the irreducible cyclic subgroups of order $r$ in ${\rm GL}(\alpha, s)$ lie in a single conjugacy class. \end{proof}

\begin{prop}
\label{GL_conj_sub_r}
The number of conjugacy classes of subgroups that are maximal amongst subgroups of ${\rm GL}(\alpha, s)$ which are also in $\mathfrak{A}_r$ is at most $1$.
\end{prop}
\begin{proof}
Let $G$ be maximal amongst subgroups of ${\rm GL}(\alpha, s)$ which are also in $\mathfrak{A}_r$.   As $\mbox{char}(\mathbb{F}_{p}) = t \nmid |G|$, therefore, by Maschke's theorem we can find groups $G_i$ such that $G \leq G_{1} \times G_{2} \times \cdots \times G_{k} = \hat{G} \leq {\rm GL}(\alpha, s)$ where for each $i$, we have that $G_i$ is a (maximal) irreducible subgroup of ${\rm GL}(\alpha_i, s)$ that is also in $\mathfrak{A}_r$. Further $\alpha = \alpha_1 + \cdots + \alpha_k$. Clearly, $G_i \cong C_r$ and that $\alpha_i = d = \mbox{order } s \bmod r$ for each $i$.  Thus we must have $\alpha = dk$ and by maximality of $G$ we get that $G = \hat{G}$. Further the conjugacy classes of $G_i$ in ${\rm GL}(\alpha_i, s)$ determine the conjugacy class of $G$ in ${\rm GL}(\alpha, s)$.

So if $d$ does not divide $\alpha$ then ${\rm GL}(\alpha, s)$ cannot have any elementary abelian $r$-subgroup. If $d \mid \alpha$ then any $G$ that is maximal amongst subgroups of ${\rm GL}(\alpha, s)$ which are also in $\mathfrak{A}_r$ must have order $r^{k}$ where $k = \alpha/d$. Then by Lemma \ref{GL_conj_irr_sub_r} clearly, all such groups form a single conjugacy class. \end{proof}

\section{Subgroups of ${\rm GL}(\alpha, s)$ which are also in $\mathfrak{A}_q\mathfrak{A}_r$}
\label{Genlinsub}
We prove results which give us a bound on the order of subgroups of ${\rm GL}(\alpha, s)$ which are in $\mathfrak{A}_q\mathfrak{A}_r$ and also a bound for the number of conjugacy classes of subgroups that are maximal amongst subgroups of ${\rm GL}(\alpha, s)$ which are in $\mathfrak{A}_q\mathfrak{A}_r$. Theorem D is proved here.
\begin{prop}
\label{GL_ord_sub_AqAr}
 Let $G$ be a subgroup of ${\rm GL}(\alpha, s)$ which is in $\mathfrak{A}_q\mathfrak{A}_r$.  
 \begin{enumerate}[\rm{(}i\rm{)}]
     \item Let $m = |F(G)|$. If $G$ is primitive then $|G| \leq cm$  where $c = \text{order $s \bmod m$}$ and $c \mid \alpha$. Further $m$ is either $r$ or $q$ or $qr$.
     \item $|G| \leq {(6^{1/2})}^{\alpha - 1} \, {d}^{\alpha}$ where $d = \mbox{min}\{  qr, s \}$.
 \end{enumerate}  
\end{prop}

\begin{proof}
Let $V = (\mathbb{F}_{s})^{\alpha}$. Let $G$ be a primitive subgroup of ${\rm GL}(\alpha, s)$ which is in $\mathfrak{A}_q\mathfrak{A}_r$ and let $|G| = q^{\beta} r^{\gamma}$ where $\beta$ and $\gamma$ are natural numbers. If $\beta = 0$ or $\gamma = 0$, then we get the required result by Lemma \ref{GL_conj_irr_sub_r}. Assume that $\beta$ and $\gamma$ are at least 1. Let $F = F(G)$ be the Fitting subgroup of $G$. Since $G \in \mathfrak{A}_q\mathfrak{A}_r$ we have that $F$ is abelian and $|F|=q^{\beta}r^{\gamma_1} =m$ where $\gamma_1 \leq \gamma$.  By Clifford's theorem, since $G$ is primitive we get that as an $F$-module, $V = X_{1} \oplus X_{2} \oplus \cdots \oplus X_{a}$ where $X_{i}$ are conjugates of $X$, an irreducible $\mathbb{F}_{s}F$-submodule of $V$. Note that $F$ acts faithfully on $X$.

Let $E$ be the subalgebra generated by $F$ in $\mbox{End}(V)$. Since, the $X_{i}$ are conjugates of $X$, therefore $E$ acts faithfully and irreducibly on $X$ and $E$ is commutative. So by \cite[Proposition 8.2 and Theorem 8.3]{SPG2007} we get that $E$ is a field. Thus $E \cong \mathbb{F}_{s^{c}}$, where $c = \dim(X)$ as a $\mathbb{F}_{s}F$-module and $\alpha = ac$. Note that $F$ is an abelian group of order $m$ acting faithfully and irreducibly on $X$. Consequently, $F$ is cyclic and so $c$ is the least positive integer such that $m | s^{c} -1 $. Clearly $m=q$ or $m=qr$ and so $\beta =1$.  It is not difficult to show that $G$ acts on $E$ by conjugation. Hence, there exists a homomorphism from $G$ to $\mbox{Gal}_{\mathbb{F}_{s}}(E)$. Let  $N$ be the kernel of this map. Then $N = C_{G}(E) \leq C_{G}(F) \leq F$. But $F \leq N$. Hence, $F=N$. So, $\frac{G}{F} \leq \mbox{Gal}_{\mathbb{F}_{s}}(E) \cong C_{c}$ and $|G| \leq cm$. 

Let $G$ be an irreducible imprimitive subgroup of ${\rm GL}(\alpha, s)$ which is also in $\mathfrak{A}_q\mathfrak{A}_r$. Then we get that $G \leq G_{1}\, {\rm wr} \,G_{2} \leq {\rm GL}(\alpha, s)$ where $G_{1}$ is a primitive subgroup of ${\rm GL}(\alpha_1, s)$ which is in $\mathfrak{A}_q\mathfrak{A}_r$, and the group $G_{2}$  can be regarded as a transitive subgroup of $S_{k}$ which is in $\mathfrak{A}_q\mathfrak{A}_r$. Further $\alpha = \alpha_{1} k$. By the above part, $|G_1| \leq c'm'$ where $c' = \text{order $s \bmod m'$}$ and $m' = |F(G_1)|$ is either $r$ or $q$  or  $qr$. Also $c' \mid \alpha_1$. By \cite[Proposition 2.1]{GV1997} we have that $|G_2| \leq {(6^{1/2})}^{k-1}$. Using $c' \leq 2^{c'-1} \leq {(6^{1/2})}^{c' - 1}$ we get that $|G| \leq {(6^{1/2})}^{\alpha - 1} \, {(m')}^{k}$. Since $m' \mid p^{c'} -1$ we get that ${(m')}^k \leq d^{\alpha}$ where $d = \mbox{min}\{  qr, s \}$.

Since $t$ does not divide $q$ or $r$, by Maschke's Theorem, any subgroup $G$ of ${\rm GL}(\alpha, s)$ which is in $\mathfrak{A}_q\mathfrak{A}_r$ will be completely reducible. Thus $G \leq G_{1} \times \cdots \times G_{k} \leq {\rm GL}(\alpha, s)$, where $G_{i}$'s are irreducible subgroups of ${\rm GL}(\alpha_i, s)$ which are in $\mathfrak{A}_q\mathfrak{A}_r$ and $\alpha = \alpha_{1} + \cdots + \alpha_{k}$. Hence, we get $|G| \leq {(6^{1/2})}^{\alpha - 1} \, {d}^{\alpha}$ where $d = \mbox{min}\{  qr, s \}$.\end{proof}

\begin{prop}
\label{GL_conj_sub_AqAr}
There exists constants $b$ and $c$ such that the number of conjugacy classes of subgroups that are maximal amongst irreducible subgroups of ${\rm GL}(\alpha, s)$ which are in $\mathfrak{A}_q\mathfrak{A}_r$ is at most $ 2^{(b+c) ({\alpha}^2/\sqrt{\log (\alpha)}) + (5/6) \log (\alpha) + \log(6)} \, s^{(3+c){\alpha}^2}$ provided $\alpha > 1$.
\end{prop}

\begin{proof}
Let $G$ be a subgroup of ${\rm GL}(\alpha, s)$ such that it is maximal amongst irreducible subgroups of ${\rm GL}(\alpha, s)$ which are in $\mathfrak{A}_q\mathfrak{A}_r$. Let $G = q^{\beta} r^{\gamma}$ where $\beta$ and $\gamma$ are natural numbers. If $\beta = 0$ or $\gamma = 0$, then we get the required result by Lemma \ref{GL_conj_irr_sub_r}. Assume that $\beta$ and $\gamma$ are at least 1. Let $V = (\mathbb{F}_{s})^{\alpha}$ and $F = F(G)$, the Fitting subgroup of $G$. Then $F=Q \times R_1$ where $Q$ is the unique Sylow $q$-subgroup of $G$ and $R_1 \leq R$, where $R$ is a Sylow $r$-subgroup of $G$. So $F$ is abelian and $|F|=q^{\beta}r^{\gamma_1} =m$ where $\gamma_1 \leq \gamma$. 

Using Clifford's theorem, regarding $V$ as $\mathbb{F}_{s}F$-module, we get that $V = Y_{1} \oplus Y_{2} \oplus \cdots \oplus Y_{l}$ where $Y_{i} = k X_{i}$ for all $i$, and $X_{1}, \ldots, X_{l}$ are irreducible $\mathbb{F}_{s}F$-submodules of $V$. Further, for each $i,j$ there exists $g_{ij} \in G$ such that $g_{ij}X_{i} = X_{j}$ and for $i=1, \ldots, l$, the $X_{i}$ form a maximal set of pairwise non-isomorphic conjugates. Also, the action of $G$ on the $Y_{i}$ is transitive. It is not difficult to check that $ C_{F}(Y_{i}) = C_{F}(X_{i}) = K_i$ say. Thus $F/K_i$ acts faithfully on $Y_{i}$ and when its action is restricted to $X_{i}$, it acts faithfully and irreducibly on $X_{i}$. Since, $X_{i}$ is a non-trivial irreducible faithful $\mathbb{F}_{s}F/K_i$-module and $t$ is coprime to $q$ and $r$, we get that $F/K_i$ is cyclic and $\dim_{\mathbb{F}_{s}}(X_i) = d_i$ where $d_i$ is the least positive integer such that $m_i$ divides $s^{d_i} -1$, and where $m_i$ is the order of $F/K_i$. Since the $X_i$ are conjugate we get that $\dim_{\mathbb{F}_{s}}(X_i) = d_i = d$ for all $i$.

Let $E_{i}$ be the subalgebra generated by $F/K_i$ in $\mbox{End}_{F_{s}}(Y_{i})$. Note that $E_{i}$ is commutative as $F/K_i$ is abelian. Further, $X_{i}$ is a faithful irreducible $E_{i}$-module. So, $E_{i}$ is simple and becomes a field such that $E_{i} \cong F_{s^{d}}$. We also observe that $\alpha = k l d$.

Let $k, l, d$ be fixed such that $\alpha = kld$. Now we find the choices for $F$ up to conjugacy in ${\rm GL}(V)$. Clearly,
\begin{align*}
F &\leq F/K_1 \times F/K_2 \times \cdots \times F/K_l \\
&\leq {E_1}^{*} \times {E_2}^{*} \times \cdots \times {E_l}^{*} \\
&\leq {\rm GL}(Y_{1}) \times {\rm GL}(Y_{2}) \times \cdots \times {\rm GL}(Y_{l}) \\
&\leq {\rm GL}(V)   
\end{align*}
where ${E_i}^*$ denotes the multiplicative group of the field $E_i$.
Let $E = {E_1}^{*} \times {E_2}^{*} \times \cdots \times {E_l}^{*}$. Then $|E|= (s^{d} - 1)^{l}$. Regarding $V$ as an $F_{s}E$-module, we get $V = kX_{1} \oplus kX_{2} \oplus \cdots \oplus kX_{l}$, where $E_{i}^{*}$ acts faithfully and irreducibly on $X_{i}$ and $\dim_{E_{i}}(X_{i}) = 1$, for all $i$. Further, for all $i \neq j$, $E_{i}^{*}$ acts trivially on $X_{j}$. It is not difficult to show that  there is only one conjugacy class of subgroups of type $E$ in ${\rm GL}(V)$.

So once $k,l$ and $d$ are chosen such that $\alpha = k l d$, up to conjugacy there is only one choice for $E$. Since $E$ is a direct product of $l$ isomorphic cyclic groups, any subgroup of $E$ can be generated by $l$ elements. In particular, $F$ can be generated by $l$ elements. So, the number of choices for $F$ as a subgroup of $E$ is at most $|E|^{l}= (s^{d} - 1)^{{l}^{2}}$.

Since, $G$ acts transitively on $\{Y_{1},\cdots, Y_{l}\}$, therefore, there exists a homomorphism say $\phi$ from $G$ into $S_{l}$. Let ${N} = \ker(\phi) = \{g \in G \mid gY_{i} = Y_{i}\mbox{ for all } $i$ \}$.  Clearly $F \leq {N}$ and $G/{N}$ is a transitive subgroup of $S_{l}$ which is in $\mathfrak{A}_r$. If $g \in {N}$, then we can show that $gE_{i}g^{-1}=E_{i}$. Thus there exists a homomorphism $\psi_{i}: {N} \rightarrow \mbox{Gal}_{\mathbb{F}_{s}}(E_{i})$. This induces a homomorphism $\psi$ from  $N$  to $\mbox{Gal}_{\mathbb{F}_{s}}(E_{1}) \times \mbox{Gal}_{\mathbb{F}_{s}}(E_{2}) \times \cdots \times \mbox{Gal}_{\mathbb{F}_{s}}(E_{l})$ such that $\ker(\psi) =  \cap_{i=1}^{l} N_{i} = F$ where $N_{i} = \ker(\psi_i) = C_{N}(E_{i})$. So, ${N}/F$ is isomorphic to a subgroup of $\mbox{Gal}_{\mathbb{F}_{s}}(E_{1}) \times \mbox{Gal}_{\mathbb{F}_{s}}(E_{2}) \times \cdots \times \mbox{Gal}_{\mathbb{F}_{s}}(E_{l})$. Since $\mbox{Gal}_{\mathbb{F}_{s}}(E_{i}) \cong C_{d}$, for every $i$ we get that ${N}/F$ can be generated by $l$ elements.  
 
Let $T={\rm GL}(\alpha, s)$. Let $\hat{N}=\{x \in N_{T}{(F)} \mid xY_{i}=Y_{i}, \,\text{for all $i$} \}$. Then $F \leq N \leq \hat{N} \leq N_{T}{(F)}$. We will find the choices for $N$ as a subgroup of $\hat{N}$, given that $F$ has been chosen. The group $\hat{N}$ acts by conjugation on $E_{i}$ and fixes the elements of $\mathbb{F}_s$. So, we have a homomorphism $\rho_{i}: \hat{N} \rightarrow \mbox{Gal}_{\mathbb{F}_{s}}(E_{i})$ with kernel $C_{\hat{N}}{(E_{i})}$. Define $C =  \cap_{i=1}^{l} C_{\hat{N}}{(E_{i})}$. Note that $N \cap C = F$. Also, $\hat{N}/C$ is isomorphic to a subgroup of $\mbox{Gal}_{\mathbb{F}_{s}}(E_{1}) \times \mbox{Gal}_{\mathbb{F}_{s}}(E_{2}) \times \cdots \times \mbox{Gal}_{\mathbb{F}_{s}}(E_{l})$, where each $\mbox{Gal}_{\mathbb{F}_{s}}(E_{i}) $ is isomorphic to $C_{d}$, for every $i$. So, $|\hat{N}/C| \leq d^{l}$. Clearly $C$ centralises $E_{i}$, for each $i$. Therefore, there exists a homomorphism from $C$ into ${\rm GL}_{\mathbb{E}_{i}}{(Y_{i})}$ for each $i$. Hence $C$ is isomorphic to a subgroup of ${\rm GL}_{\mathbb{E}_{i}}(Y_{1}) \times {\rm GL}_{\mathbb{E}_{i}}(Y_{2}) \times \cdots \times {\rm GL}_{\mathbb{E}_{i}}(Y_{l})$. As, $\dim_{\mathbb{E}_{i}}(Y_i) = k$  and $E_{i} \cong F_{s^{d}}$, for all $i$, therefore, $|C| \leq s^{dk^{2}l}$. Hence $|\hat{N}| \leq d^{l}s^{dk^{2}l}$. 

Now $NC/C \cong N/(N \cap C) = N/F$. So we get that $NC/C$ can be generated by $l$ elements since $N/F$ can be generated by $l$ elements. But $|\hat{N}/C| \leq d^{l}$, therefore, the choices for $NC/C$ as a subgroup of $\hat{N}/C$ is at most $d^{l^{2}}$. Once we make a choice for $NC/C$ as a subgroup of $\hat{N}/C$, we choose a set of $l$ generators for $NC/C$. As $N \cap C = F$, we get that $N$ is determined as a subgroup of $\hat{N}$ by $F$ and $l$ other elements that map to the chosen generating set of $NC/C$. We have $|C|$ choices for an element of $\hat{N}$ that maps to any fixed element of $\hat{N}/C$. Thus, there are at most $|C|^{l}$ choices for $N$ as a subgroup of $\hat{N}$ once $NC/C$ has been chosen. So we have at most $d^{l^{2}}(s^{{dk^{2}l}})^{l} = d^{l^{2}}s^{dk^{2}l^{2}}$ choices for $N$ as a subgroup of $\hat{N}$, once $F$ is fixed. 

Now we find the choices for $G$ given that $F$ and $N$ are determined and fixed as a subgroups of $T$ and $\hat{N} \leq T$ respectively. Let $\hat{Y}=\{y \in N_{T}{(F)} \mid y \mbox{ permutes the } Y_i \}$. Then $F \leq G \leq \hat{Y} \leq N_{T}{(F)} \leq {\rm GL}(V)$. Also there exists a homomorphism from $\hat{Y}$ to $S_{l}$ with kernel $\{y \in \hat{Y} \mid yY_{i} = Y_{i}, \mbox{ for all $i$}\} = \hat{N}$. Thus $\hat{Y}/\hat{N}$ may be regarded as a subgroup of $S_{l}$. But $G \cap \hat{N} = N$. Thus $G/N = G/(G \cap \hat{N}) \cong G\hat{N}/\hat{N}$. So $G/N \cong G\hat{N}/\hat{N} \leq \hat{Y}/\hat{N} \leq S_{l}$. Note that $G/N$ is a transitive subgroup of $S_{l}$ which is in \textbf{$\mathfrak{A}_{r}$}. By \cite[Theorem 1]{LMM1998}, there exists a constant $b$ such that  $S_{l}$ has at most $ 2^{bl^{2} / {\sqrt{\log (l)}}}$ transitive subgroups for $l >1$. Hence, the choices for $G\hat{N}/\hat{N}$ as a subgroup of $\hat{Y}/\hat{N}$ is at most $ 2^{bl^{2} / {\sqrt{\log (l)}}}$.

\smallskip
By \cite[Theorem 2]{LMM2000}, there exists a constant $c$ such that any transitive permutation group of finite degree greater than $1$ can be generated by $\lfloor cl/\sqrt{log(l)}\rfloor$. Thus $G\hat{N}/\hat{N}$ can be generated by $\lfloor cl/\sqrt{log(l)}\rfloor$ for $l > 1$. Once a choice for $G\hat{N}/\hat{N}$ is made as a subgroup of $\hat{Y}/\hat{N}$ and $\lfloor cl/\sqrt{log(l)}\rfloor$ generators are chosen for $G\hat{N}/\hat{N}$ in $\hat{Y}/\hat{N}$ then $G$ will be determined as a subgroup of $\hat{Y}$ by $\hat{N}$ and the choices of elements of $\hat{Y}$ that map to the $\lfloor cl/\sqrt{log(l)}\rfloor$ generators chosen for $G\hat{N}/\hat{N}$. So, we have at most $|\hat{N}|^{\lfloor cl/\sqrt{log(l)}\rfloor}$ choices for $G$ as a subgroup of $\hat{Y}$ once a choice of $G\hat{N}/\hat{N}$ in $\hat{Y}/\hat{N}$ is fixed. Hence we have
$$2^{bl^{2} / {\sqrt{\log (l)}}}\, (d^{l}s^{dk^{2}l})^{\lfloor cl/\sqrt{\log(l)}\rfloor} \leq  2^{bl^{2} / {\sqrt{\log (l)}}}\, 
d^{cl^{2}/ \sqrt{\log(l)}}\,  s^{{cdk^{2}l^{2}}/\sqrt{\log(l)}}$$ choices for $G$ as a subgroup of $\hat{Y}$ assuming that choices for $F$ and $N$ have been made. Putting together all the above estimates we get that the 
number of conjugacy classes of subgroups that are maximal amongst irreducible subgroups of ${\rm GL}(\alpha, s)$ that are in $\mathfrak{A}_q\mathfrak{A}_r$ is at most 
\begin{align*}
\sum_{(k,l,d)}^{}  {(s^{d} - 1)^{{l}^{2}} d^{l^{2}} s^{dk^{2}l^{2}} {2^{bl^{2} / {\sqrt{\log (l)}}}} {d^{cl^{2}/ \sqrt{\log(l)}}}  {s^{{cdk^{2}l^{2}}/\sqrt{\log(l)}}}}.
\end{align*}
where $(k,l,d)$ ranges over ordered triples of natural numbers which satisfy $\alpha = k l d$ and $l >1$. We simplify the above expression as follows. Using $\alpha = k l d$ we get
$$
{(s^{d} - 1)^{{l}^{2}}d^{l^{2}}s^{dk^{2}l^{2}}{s^{{cdk^{2}l^{2}}/\sqrt{\log(l)}}}} \leq s^{(3 +c){\alpha}^2} .
$$
Since $x/\sqrt {\log (x)}$ is increasing for $x > e^{1/2}$, we have $l/\sqrt{\log (l)} \leq \alpha/\sqrt{\log (\alpha)}$ for $l \geq 2$. Thus we get that $2^{bl^{2} / {\sqrt{\log (l)}}} d^{cl^{2}/ \sqrt{\log(l)}}  \leq 2^{(b + c) {\alpha}^2/\sqrt{\log (\alpha)}}$. 

\smallskip
There are at most $2^{\frac{5}{6}\log (\alpha) +\log(6)}$ choices for $(k,l,d)$.  Thus we get that there exists constant $b$ and $c$ such that the number of conjugacy classes of subgroups that are maximal amongst irreducible subgroups of ${\rm GL}(\alpha, s)$ that are in $\mathfrak{A}_q\mathfrak{A}_r$ is at most  $ 2^{(b+c) ({\alpha}^2/\sqrt{\log (\alpha)}) + (5/6) \log (\alpha) + \log(6)} \, s^{(3+c){\alpha}^2}$ provided $\alpha > 1$. \end{proof}

Theorem D follows as a corollary to Proposition \ref{GL_conj_sub_AqAr}. The proof is given below.
\newpage 
\noindent {\bf{Proof of Theorem D}}
\begin{proof}
Let $G$ be maximal amongst subgroups of ${\rm GL}(\alpha, s)$ which are also in $\mathfrak{A}_q\mathfrak{A}_r$. As characteristic of $\mathbb{F}_{s} = t$ and $t \nmid |G|$, therefore, by Maschke's theorem, we have that $G \leq \hat{G_{1}}\times \cdots \times \hat{G_{k}} \leq {\rm GL}(\alpha, s)$ where $\hat{G_{i}}$ are maximal among irreducible subgroups of ${\rm GL}(\alpha_{i}, p)$ which are also in $\mathfrak{A}_q\mathfrak{A}_r$, and where $\alpha = \alpha_{1} + \cdots + \alpha_{k}$. By maximality of $G$, we have $G =  \hat{G_{1}}\times \cdots \times \hat{G_{k}}$. 

Further, the conjugacy classes of $\hat{G_{i}} \in {\rm GL}(\alpha_{i}, s)$ determine the conjugacy class of $G \in {\rm GL}(\alpha, s)$. So, the number of conjugacy classes of subgroups that are maximal amongst the subgroups of ${\rm GL}(\alpha, s)$ which are also in $\mathfrak{A}_q\mathfrak{A}_r$ is at most 
$ \sum \prod_{i=1}^{k}  2^{(b+c) ({\alpha_i}^2/\sqrt{\log (\alpha_i)}) + (5/6) \log (\alpha_i) + \log(6)} \, s^{(3+c){\alpha_i}^2}$ by Proposition \ref{GL_conj_sub_AqAr}, and where the sum is over all unordered partitions $\alpha_{1}, \cdots, \alpha_{k}$ of $\alpha$. We assume that if $\alpha_i =1$ for some $i$, then the part of expression corresponding to it in the product is 1. Since $x/\sqrt {\log (x)}$ is increasing for $x > e^{1/2}$, and $\alpha = \alpha_{1} + \cdots + \alpha_{k}$ we get that
$$
\prod_{i=1}^{k}  2^{(b+c) ({\alpha_i}^2/\sqrt{\log (\alpha_i)}) + (5/6) \log (\alpha_i) + \log(6)} \leq 2^{(b+c) ({\alpha}^2/\sqrt{\log (\alpha)}) + (5/6) \alpha \log (\alpha) + \alpha \log(6)}.
$$

It is not difficult to show that the number of unordered partitions of $\alpha$ is at most $2^{\alpha -1}$.  So, the number of conjugacy classes of subgroups that are maximal amongst the subgroups of ${\rm GL}(\alpha, s)$ which are also in $\mathfrak{A}_q\mathfrak{A}_r$ is at most 
$$
2^{(b+c) ({\alpha}^2/\sqrt{\log (\alpha)}) + (5/6) \alpha \log (\alpha) + \alpha (1+ \log(6))} s^{(3+c){\alpha}^2}
$$
provided $\alpha > 1$. \end{proof}

We end this section with the following remark which provides an alternate bound.
\begin{rema}
\label{Thmd} We do not have an estimate for the constants $b$, $c$ occurring in Theorem D. If we use a weaker fact that any subgroup of $S_n$ can be generated by $\lfloor n/2 \rfloor$ elements for all $n \geq 3$, then we get a weaker result for the number of transitive subgroups of $S_n$ that are in $\mathfrak{A}_q\mathfrak{A}_r$, namely that they are at most $6^{n(n-1)/4}\, 2^{(n+2) \log (n)}$. Using this in the proof of Theorem D, we get that the number of conjugacy classes of subgroups that are maximal amongst the subgroups of ${\rm GL}(\alpha, s)$ which are also in $\mathfrak{A}_q\mathfrak{A}_r$ is at most $$s^{5\alpha^{2}} \, 6^{\alpha(\alpha-1)/4} \,2^{\alpha - 1 + (23/6) \alpha \log(\alpha) + \alpha \log(6)}$$ where $t, q$ and $r$ are distinct primes, $s$ is a power of $t$ and $\alpha \in \mathbb{N}$.
\end{rema}

\section{Enumeration of groups in $\mathfrak{A}_p\mathfrak{A}_q \mathfrak{A}_r$}
\label{Enu_pqr}

In this section we prove Theorem A, namely,
$$f_{\mathfrak{S}}(n) \leq p^{6\alpha^{2}} \,2^{\alpha - 1 + (23/6) \alpha \log(\alpha) + \alpha \log(6)}\, (6^{1/2})^{{({\alpha + \gamma})\beta} + {({\alpha + \beta})\gamma} +{\alpha(\alpha-1)/2}}\, n^{\beta + \gamma},$$ where $n=p^{\alpha}q^{\beta}r^{\gamma}$ and $\alpha, \beta, \gamma \in \mathbb{ N}$. We use techniques adapted from \cite{LP1993}, \cite{GV1997} and \cite{GV1999}.

\medskip
\noindent {\bf{Proof of Theorem A}}
\begin{proof}
 Let $G$ be a group of order $n = p^{\alpha}q^{\beta}r^{\gamma}$ in $\mathfrak{A}_p\mathfrak{A}_q\mathfrak{A}_r$. Then $G= P \rtimes H$ where $P$ is the unique Sylow $p$-subgroup of $G$ and $H \in \mathfrak{A}_q\mathfrak{A}_r$. So we can write $H =Q \rtimes R$ where $|Q|=q^{\beta}$ and $|R|=r^{\gamma}$. Let $G_{1}={G}/{O_{p'}{(G)}}$, $G_{2}={G}/{O_{q'}{(G)}}$ and $G_{3}={G}/{O_{r'}{(G)}}$.  Clearly each $G_{i}$ is a soluble $A$-group and $G \leq G_{1} \times G_{2} \times G_{3}$ as a subdirect product. Further, ${O_{p'}{(G_{1})}}=1={O_{q'}{(G_{2})}}={O_{r'}{(G_{3})}}$.

Since $G_{1}={G}/{O_{p'}{(G)}}$ we get that $G_{1} \in \mathfrak{A}_p\mathfrak{A}_q\mathfrak{A}_r$ and if $P_1$ is the Sylow $p$-subgroup of $G_1$ then $P_{1} \cong P$. Thus, $|G_1|=p^{\alpha}q^{\beta_{1}}r^{\gamma_{1}}$ and we can write $G_{1}= P_{1} \rtimes H_{1}$ where $H_{1} \in \mathfrak{A}_q\mathfrak{A}_r$. So $H_{1}= Q_{1} \rtimes R_{1}$ where $Q_{1}\in \mathfrak{A}_q$ and $|Q_1|=q^{\beta_{1}}$, $R_{1}\in \mathfrak{A}_r$ and $|R_1|=r^{\gamma_{1}}$. Further, $H_{1}$ acts faithfully on $P_{1}$. Hence we can regard $H_1 \leq \rm Aut(P_{1}) \cong$ $ {\rm GL}(\alpha, p)$. Let $M_{1}$ be a subgroup that is maximal amongst $p'$-$A$-subgroups of ${\rm GL}(\alpha, p)$ that are also in $\mathfrak{A}_q\mathfrak{A}_r$ and such that $H_{1} \leq M_{1}$. Let $\hat{G_{1}}=P_{1}M_{1}$. The number of conjugacy classes of the $M_{1}$ in ${\rm GL}(\alpha, p)$ is at most 
 $p^{5\alpha^{2}} \, 6^{\alpha(\alpha-1)/4} \,2^{\alpha - 1 + (23/6) \alpha \log(\alpha) + \alpha \log(6)}$ \mbox{by Remark \ref{Thmd}}.

Since $G_{2}={G}/{O_{q'}{(G)}}$ we will get that $G_{2} \in \mathfrak{A}_q\mathfrak{A}_r$ and if $Q_2$ is the Sylow $q$-subgroup of $G_2$ then $Q_{2} \cong Q$. Thus, $|G_2|=q^{\beta}r^{\gamma_{2}}$ and we can write $G_{2}=Q_{2} \rtimes H_{2}$ where $H_{2} \in \mathfrak{A}_r$. So $|H_2|=r^{\gamma_{2}}$. Also $H_2 \leq \rm Aut(Q_{2}) \cong $ ${\rm GL}(\beta, q)$. Let $M_{2}$ be a subgroup that is maximal amongst $q'$-$A$-subgroups of ${\rm GL}(\beta, q)$ that are also in $\mathfrak{A}_r$ and such that $H_{2} \leq M_{2}$. Let $\hat{G_{2}} = Q_{2}M_{2}$. The number of conjugacy classes of the $M_{2}$ in ${\rm GL}(\beta, q)$ is at most $1$ \mbox{by Proposition \ref{GL_conj_sub_r}}.

Since $G_{3}={G}/{O_{r'}{(G)}}$ we will get that $G_{3} \in \mathfrak{A}_r\mathfrak{A}_q$ and if $R_3$ is the Sylow $r$-subgroup of $G_3$ then $R_{3} \cong R$. Thus, $|G_3|=q^{\beta_3}r^{\gamma}$ and we can write $G_{3}=R_{3} \rtimes H_{3}$ where $H_{3} \in \mathfrak{A}_r$. So $|H_3|=q^{\beta_{3}}$. Also $H_3 \leq \rm Aut(R_{3}) \cong$ $ {\rm GL}(\gamma, r)$. Let $M_{3}$ be a subgroup that is maximal amongst $r'$-$A$-subgroups of ${\rm GL}(\gamma, r)$ that are also in $\mathfrak{A}_q$ and such that $H_{3} \leq M_{3}$. Let $\hat{G_{3}} = R_{3}M_{3}$. The number of conjugacy classes of the $M_{3}$ in ${\rm GL}(\gamma, r)$ is at most $1$ \mbox{by Proposition \ref{GL_conj_sub_r}}.

Let $\hat{G} = \hat{G_{1}} \times \hat{G_{2}} \times \hat{G_{3}}$. Then $G \leq \hat{G}$. We know that the choices for $P_{1}, Q_{2}$ and $R_{3}$ is unique, up to isomorphism. We enumerate the possibilities for $\hat{G}$ up to isomorphism and then find the number of subgroups of $\hat{G}$ of order $n$ up to isomorphism. For the former we count the number of $\hat{G_{i}}$ up to isomorphism which depends on the conjugacy class of the $M_{i}$ in $A_{i}$. Hence, the number of choices for $\hat{G}$ up to isomorphism $ = \displaystyle \prod_{i=1}^{3} \text{Number of choices for $\hat{G_{i}}$ up to isomorphism}$.
Now we estimate the choices for $G$ as a subgroup of $\hat{G}$ using a method of `Sylow systems' introduced by Pyber in \cite{LP1993} .

Let $\hat{G}$ be fixed. We now count the number of choices for $G$ as a subgroup of $\hat{G}$. Let ${\cal S} = \{S_{1}, S_{2}, S_{3}\}$ be a Sylow system for $G$ where $S_{1}$ is the Sylow $p$-subgroup of $G$, $S_{2}$  is a Sylow $q$-subgroup of $G$ and $S_{3}$ is a Sylow $r$-subgroup of $G$ such that $S_{i}S_{j} = S_{j}S_{i}$ for all $i, j=1,2,3$. Then $G = S_{1} S_{2} S_{3}$. By \cite[Theorem 6.2, Page-49]{SPG2007}, we know that there exists ${\cal B} = \{B_{1}, B_{2}, B_{3}\}$, a Sylow system for $\hat{G}$ such that $S_{i} \leq B_{i}$ where $B_{1}$ is the Sylow $p$-subgroup of $\hat{G}$, $B_{2}$ is a Sylow $q$-subgroup of $\hat{G}$ and $B_{3}$ is a Sylow $r$-subgroup of $\hat{G}$. Note that $|B_1|=p^{\alpha}$. Further any two Sylow systems for $\hat{G}$ are conjugate. Hence, the number of choices for $G$ as a subgroup of $\hat{G}$ and up to conjugacy is at most
$$\mid \!\{S_1, S_2, S_3 \mid S_{i} \leq B_{i}, |S_1|=p^{\alpha}, |S_2|=q^{\beta}, |S_3|=r^{\gamma}\! \} \mid \,\leq \,  |B_1|^{\alpha} |B_2|^{\beta} |B_3|^{\gamma}.  $$

 We observe that $B_2= T_{21} \times T_{22} \times T_{23}$ where $T_{2i}$ are some Sylow $q$-subgroups of $\hat{G_{i}}$ for $i=1,2,3$. Using \cite[Proposition 3.1]{GV1997}, we get that $|T_{21}| \leq |M_1| \leq (6^{1/2})^{{\alpha}-1} p^{\alpha}$ and $|T_{23}| = |M_3| \leq (6^{1/2})^{{\gamma}-1} r^{\gamma}$. Further $|T_{22}|=|Q_2|=q^{\beta}$. Hence, we get that $|B_2| \leq  (6^{1/2})^{{\alpha + \gamma}-2} p^{\alpha} q^{\beta} r^{\gamma} \leq (6^{1/2})^{{\alpha + \gamma}} n$, and so $|B_2|^{\beta}  \leq  (6^{1/2})^{({\alpha + \gamma})\beta} n^{\beta}$. Similarly we can show that $|B_3| \leq  (6^{1/2})^{{\alpha + \beta}-2} p^{\alpha} q^{\beta} r^{\gamma}$. So $|B_3|^{\gamma}  \leq  (6^{1/2})^{({\alpha + \beta})\gamma} n^{\gamma}$. Now we put all the estimates together to get that the number of choices for $G$ as a subgroup of $\hat{G}$ up to conjugacy is at most $|B_1|^{\alpha} |B_2|^{\beta} |B_3|^{\gamma}$ which is less than or equal to
$$p^{{\alpha}^2}\, (6^{1/2})^{({\alpha + \gamma})\beta} n^{\beta} (6^{1/2})^{({\alpha + \beta})\gamma} n^{\gamma} \leq p^{{\alpha}^2}\, (6^{1/2})^{{({\alpha + \gamma})\beta} + {({\alpha + \beta})\gamma}}n^{\beta + \gamma}.$$
Therefore, the number of groups of order $p^{\alpha}q^{\beta}r^{\gamma}$ in $\mathfrak{A}_p\mathfrak{A}_q\mathfrak{A}_r$ up to isomorphism is 
\begin{eqnarray*}
&\leq& p^{5\alpha^{2}} \, 6^{\alpha(\alpha-1)/4} \,2^{\alpha - 1 + (23/6) \alpha \log(\alpha) + \alpha \log(6)}\,
p^{{\alpha}^2}\, (6^{1/2})^{{({\alpha + \gamma})\beta} + {({\alpha + \beta})\gamma}}\, n^{\beta + \gamma} \\
&=& p^{6\alpha^{2}} \,2^{\alpha - 1 + (23/6) \alpha \log(\alpha) + \alpha \log(6)}\, (6^{1/2})^{{({\alpha + \gamma})\beta} + {({\alpha + \beta})\gamma} +{\alpha(\alpha-1)/2}}\, n^{\beta + \gamma}.
\end{eqnarray*}
\end{proof}

\end{document}